\documentclass[11pt]{amsart}
\usepackage{amssymb,amsmath,amsthm}
\usepackage{enumerate, url}
\newtheorem{theorem}{Theorem}
\newtheorem{lemma}[theorem]{Lemma}
\newtheorem{corollary}[theorem]{Corollary}

\newtheorem{proposition}[theorem]{Proposition}

\theoremstyle{remark}

\newtheorem*{remark}{Remark}

\numberwithin{theorem}{section} \numberwithin{equation}{section}

\newcommand{\mfa}{\mathfrak{a}}
\newcommand{\mff}{\mathfrak{f}}
\newcommand{\mfd}{\mathfrak{d}}

\newcommand{\calN}{\mathcal{N}}

\newcommand{\mfp}{\mathfrak{p}}
\newcommand{\mfm}{\mathfrak{m}}

\newcommand{\C}{\mathbb{C}}

\newcommand{\Q}{\mathbb{Q}}

\newcommand{\N}{\mathcal{N}}

\newcommand{\GL}{{\text {\rm GL}}}
\newcommand{\Irr}{{\text {\rm Irr}}}

\newcommand{\Gal}{{\text {\rm Gal}}}

\newcommand{\rd}{\textnormal{rd}}
\newcommand{\Disc}{\textnormal{Disc}}
\newcommand{\Cunc}{C_1}
\newcommand{\Ccond}{C_2}
\newcommand{\Ccor}{C_3}
\newcommand{\CEV}{C_4}
\newcommand{\Cxbound}{C_{16}}
\newcommand{\Crepexponent}{C_{17}}
\newcommand{\CIKerror}{C_{18}}
\newcommand{\CIKerrorb}{C_{19}}

\begin{document}
\title[Galois Representations for Fields of Small Discriminant]
{On the existence of large degree Galois representations for fields of small discriminant}

\author{Jeremy Rouse}
\address{Department of Mathematics, Wake Forest University, Winston-Salem, NC 27109}
\email{rouseja@wfu.edu}

\author{Frank Thorne}
\address{Department of Mathematics, University of South Carolina,
1523 Greene Street, Columbia, SC 29208}
\email{thorne@math.sc.edu}

\begin{abstract}

  Let $L/K$ be a Galois extension of number fields. We prove two lower
  bounds on the maximum of the degrees of the irreducible complex
  representations of $\Gal(L/K)$, the sharper of which is conditional
  on the Artin Conjecture and the Generalized Riemann Hypothesis. Our
  bound is nontrivial when $[K : \Q]$ is small and $L$ has small root
  discriminant, and might be summarized as saying that such fields
  can't be ``too abelian.''

\end{abstract}
\maketitle

\section{Introduction}
It is known that the discriminant of a number field cannot be too
small. Minkowski's work on the geometry of numbers implies that
$|\Disc(K)| > (\frac{e^2 \pi}{4} - o(1))^{[K : \Q]}$; we write this
bound as $\rd_K > \frac{e^2 \pi}{4} - o(1)$, where $\rd_K :=
(|\Disc(K)|)^{1/[K:\Q]}$ is the {\itshape root discriminant} of
$K$. These bounds can be improved by using analytic properties of the
Dedekind zeta function of $K$, and this was noticed by Stark (see the
parenthetical comment in the proof of Lemma 4 on page 140 of
\cite{Stark}), and worked out in detail by Andrew Odlyzko \cite{O1976}
in his MIT dissertation (supervised by Stark). The
sharpest known bounds, due to Poitou \cite{poitou} (see also Odlyzko
\cite{O}), are
\begin{equation}\label{eq:p_disc_bounds}
\rd_K \geq (60.8395\dots)^{r_1/[K:\Q]} (22.3816\dots)^{2 r_2/[K:\Q]} - O([K:\Q]^{-2/3}),
\end{equation}
where $[K : \Q] = r_1 + 2 r_2$, and
$r_{1}$ and $r_{2}$ are the numbers of real and complex embeddings of $K$,
respectively.
(The error term in \eqref{eq:p_disc_bounds} can be improved.)
If one assumes the generalized Riemann Hypothesis, the constants above can
be improved to $215.3325\dots$ and $44.7632\dots$ respectively.

Conversely, Golod and Shafarevich \cite{GS} proved that these bounds
are sharp apart from the constants, by establishing the existence of
{\itshape infinite class field towers} $K_1 \subseteq K_2 \subseteq
K_3 \subseteq$ where each $K_{i + 1}/K_i$ is abelian and unramified,
so that each field $K_i$ has the same root discriminant.  Martinet
\cite{martinet} gave the example $K_1 = \Q(\zeta_{11} +
\zeta_{11}^{-1}, \sqrt{-46})$, which has an infinite $2$-class field
tower of root discriminant $92.2\dots$, and Hajir and Maire \cite{HM1,
  HM2} constructed a tower of fields with root discriminants bounded
by $82.2$, in which tame ramification is allowed.

It is expected that fields of small discriminant should be uncommon.
For example, in \cite{O} Odlyzko asks whether there are infinitely
many fields of {\itshape prime degree} of bounded root discriminant;
such fields cannot be constructed via class field towers.  Several
researchers have studied this question in small degree.  Jones and
Roberts \cite{JR} studied the set of Galois number fields $K/\Q$ with
certain fixed Galois groups $G$; for a variety of groups including
$A_4, \ A_5, \ A_6, \ S_4, \ S_5, \ S_6$ they proved that $\rd_K >
44.7632\ldots$ apart from a finite list of fields $K$ which they
compute explicitly. Voight \cite{V} studied the set of all totally
real number fields $K$ with $\rd_K \leq 14$, finding that there are
exactly 1229 such fields, each with $[K : \Q] \leq 9$.

In light of this work, it is natural to ask whether Galois extensions
of small absolute discriminant must have any special algebraic properties. (The
analogous problems for non-normal extensions are much more delicate.)
The easiest result to prove is that they {\itshape cannot be abelian}, and
we carry this out over $\Q$ in the introduction (starting with
\eqref{eqn:kw}). In \cite{Leshin}, it is proven that given a number
field $K$, a positive integer $n$, and real number $N$, there are only
finitely many Galois extensions $L/K$ with $\Gal(L/K)$ solvable with
derived length $\leq n$ and with $\rd_{L} \leq N$. In this paper, we
study the representation theory of Galois groups of extensions of
small discriminant, and prove that such Galois groups must have
(relatively) large degree complex representations.

We will prove two versions of this result. The first is the following:
\begin{theorem}\label{thm_uncond}
Let $L/K$ be a Galois extension and let $r$ be the maximum of the degrees of the irreducible
complex representations of $\Gal(L/K)$. Then, there is a constant
$\Cunc$ so that
\begin{equation}
r \geq \frac{1}{\log(\rd_{L})} \left(\Cunc \frac{\log \log [L : \Q]}{\log \log \log [L : \Q]} - \log [K : \Q]\right).
\end{equation}
\end{theorem}

\begin{remark}
The bound of course only makes sense for large $[L : \Q]$.
A straightforward but somewhat lengthy calculation shows that
we may take $\Cunc = \frac{1}{16}$ provided $[L : \Q] \geq e^{e^{8}}$.
\end{remark}

The basic idea of the proof is to regard $L$ as an abelian extension
of an intermediate field $F$ of small degree. The existence of such an
$F$ follows from Theorem 12.23 of \cite{I}, which states that if $G$
is a finite group with the property that all of its irreducible
representations have degree $\leq r$, then $G$ must have an abelian
subgroup of index $\leq (r!)^{2}$. (There is also a converse given in
Problem 2.9 of \cite{I}: if $G$ has an irreducible representation of
degree $> r$, then $G$ cannot have an abelian subgroup of index $\leq
r$.)
We may then adapt our proof for
the abelian case to prove that either $L$ has small root discriminant,
or $F$ has relatively large degree.

It is also possible to study the representations of $\Gal(L/\Q)$
directly, without first passing to an intermediate extension $F$, via
Artin $L$-functions. We were unable to improve upon Theorem \ref{thm_uncond}
this way, but under the hypothesis that Artin $L$-functions are well
behaved we prove the following improvement of Theorem
\ref{thm_uncond}:

\begin{theorem}\label{thm_main}
  Assume that all Artin $L$-functions are entire and satisfy the
  Riemann Hypothesis.  There is a positive constant $\Ccond$ so that if
  $L/K$ is any Galois extension of number fields of degree $d$, then
  $\Gal(L/K)$ must have an irreducible complex representation of
  degree
\begin{equation}\label{eqn_main_thm}
\geq \frac{\Ccond(\log [L : \Q])^{1/5}}{(\log \rd_L)^{2/5} [K : \Q]^{3/5}}.
\end{equation}
\end{theorem}

\begin{remark}
Two issues arise when attempting to prove an unconditional version of the
above result. The first is that the unconditional zero-free regions
for $L$-functions have implied constants that depend quite badly on the
degree of the $L$-function involved. (See for example
Theorem~5.33 of \cite{IK}.) The second is the presence of the possible
exceptional zero. Without accounting for the exceptional zero issue,
it seems that the best lower bound we can obtain using the zero-free regions
mentioned above is $r \gg \sqrt{\log \log [L : \Q]}$, for an implied constant depending on $K$ and on
$\rd_L$.
\end{remark}

We now illustrate the nature of our question by handling the case where $L/\Q$ is abelian of degree $> 2$. By
Kronecker-Weber we have $L \subseteq \Q(\zeta_n)$ for some $n$, and
\begin{equation}\label{eqn:kw}
\zeta_{L/\Q}(s) = \zeta(s) \prod_{i = 2}^{[L : \Q]} L(s, \chi_i),
\end{equation}
where $\chi_i$ are Dirichlet characters of conductor $N_i$ for some $N_i | n$.
We have $\Disc(L) = \prod_i N_i$, and therefore
\begin{equation}\label{eq:abelian_rd_bound}
\log(\rd_L)  = \frac{1}{[L : \Q]} \sum_{i = 2}^{[L:\Q]} \log(N_i).
\end{equation}
Let $M := \sqrt{[L : \Q]/2}.$ There are $\leq M^2$ Dirichlet characters with conductor $\leq M$, so that with $[L : \Q] = 2M^2$
the right side of \eqref{eq:abelian_rd_bound} is $> \frac{1}{2} \log M$, so that
\begin{equation}\label{eqn_simple}
\rd_L > \exp\Big( \frac{1}{4} \log\Big( \frac{[L : \Q]}{2} \Big) \Big),
\end{equation}
a bound of the same shape as our theorems.
Although this proof is not complicated, it makes essential use of
class field theory and it seems that the use of sophisticated tools
cannot be avoided.  We could improve our bound somewhat, but note that
it is already stronger than the (conditional) bound $\log(\rd_L) \geq C_2
(\log [L:\Q])^{1/5}$ implied by Theorem \ref{thm_main}.
Observe also that for $L = \Q(\zeta_p)$ we have $\rd_L = p^{\frac{p -
    2}{p - 1}}$ and $[L:\Q] = p - 1 \approx \rd_L$, implying a limit on the scope
    for improvement.

    As an application of Theorem~\ref{thm_main} we can say something
    about unramified extensions of a fixed number field $K$. Of
    course, the maximal unramified \emph{abelian} extension of $K$ is
    the Hilbert class field of $K$ and the degree of this extension is
    $h_{K}$, the order of the ideal class group. However, there are
    number fields $K$ with Galois extensions $L$ unramified at all
    finite primes so that $\Gal(L/K)$ has no non-trivial abelian
    quotients. One of Artin's favorite examples is $K =
    \Q(\sqrt{2869})$, where if $L$ is the splitting field of $x^{5} - x -
    1$ over $\Q$, then $L/K$ is unramified and $\Gal(L/K) \cong
    A_{5}$.

\begin{corollary}
\label{unramified}
Assume that all Artin $L$-functions are entire and satisfy the Riemann
Hypothesis. Let $L_{1}/K$, $L_{2}/K$, $\ldots$, $L_{N}/K$ be linearly
disjoint unramified Galois extensions and suppose that $\Gal(L_{i}/K)$ has
an irreducible representation of degree $r$ for $1 \leq i \leq N$. Then
there is a constant $\Ccor$ so that
\[
  \log(N) \leq \Ccor r^{5} \log^{2}(|\Disc(K)|) [K : \Q].
\]
\end{corollary}
\begin{remark}
The main theorem proven by Ellenberg and Venkatesh in \cite{EV} shows that
the number $M$ of degree $n$ unramified extensions of $K$ satisfies
\[
  \log(M) \ll_{\epsilon} n^{\epsilon} (n \log(|\Disc(K)|) + \CEV [K:\Q] )
\]
for a constant $\CEV$ depending on $n$.
Because the power of $\log(|\Disc(K)|)$ is
smaller, this result is better for fixed $n$ and varying $K$. However,
since the size of $\CEV$ is not specified, our result is better for
fixed $K$ and varying $n$.
\end{remark}

\begin{remark}
  Another potential application occurs in the case when $r = 2$ and $K
  = \Q$.  Our theorem gives bounds on the number of degree $2$ Artin
  $L$-functions with conductor bounded by $q$. In the odd case, these
  arise from weight $1$ newforms of level $q$ and in the even case,
  these arise (conjecturally) from Maass forms with eigenvalue
  $1/4$. However, we obtain bounds that are worse than polynomial in $q$. In
  \cite{MV}, Michel and Venkatesh use the Petersson-Kuznetsov formula
  to obtain bounds of the form $q^{c + \epsilon}$, where $c$ is a
  constant depending on the type of representation (dihedral,
  tetrahedral, octohedral, or icosahedral).
\end{remark}

Throughout the paper we use the notation $|\Disc(K)|$ for the absolute
value of the discriminant of $K$, $\mathcal{O}_{K}$ for the ring of
integers of $K$, $\rd_{K} = |\Disc(K)|^{1/[K : \Q]}$,
$\mathcal{N}_{K/\Q}(\mathfrak{a})$ for the norm from $K$ to $\Q$ of an
ideal of $\mathcal{O}_{K}$, $h_{K}$ for the order of the ideal class
group of $\mathcal{O}_{K}$, and $\mfd_{L/K}$ for the relative
discriminant of $L$ over $K$. We denote by
$C_{1}$, $C_{2}$, $\ldots$, a sequence of absolute constants. We also
occasionally write $f \ll g$ to mean $f \leq C g$ for some constant $C$,
absolute unless otherwise noted.

 We
provide a little bit of preliminary background in
Section~\ref{sec_background}, and then we prove
Theorem~\ref{thm_uncond} in Section~\ref{sec_elem} and
Theorem~\ref{thm_main} and Corollary~\ref{unramified} in
Section~\ref{sec_main}.

\section*{Acknowledgments}
We would like to thank Kannan Soundararajan for suggesting the basic idea of this paper to us. We'd also like to thank Charlotte Euvrard and 
the referee for a number of very helpful
suggestions.

\section{Background on number fields, discriminants, and conductors}\label{sec_background}
In this section we briefly recall a few facts related to zeta and $L$-functions
associated to number fields, used in the proofs of both Theorem~\ref{thm_uncond}
and Theorem~\ref{thm_main}.

The {\itshape Dedekind zeta function} of a number field $L$ is given by the Dirichlet series
\begin{equation}
\zeta_L(s) = \sum \calN_{L/\Q}(\mfa)^{-s},
\end{equation}
where the sum is over integral ideals of $\mathcal{O}_{L}$. For a Galois extension $L/K$, this zeta function
enjoys the factorization
\begin{equation}\label{eq:ded_non_abelian}
\zeta_L(s) = \zeta_{L/K} (s) =  \prod_{\rho \in \Irr(\Gal(L/K))} L(s, \rho)^{\deg \rho},
\end{equation}
where $\rho$ varies over all irreducible complex representations of $G := \Gal(L/K)$,
and $L(s, \rho)$ is the associated Artin $L$-function.
(For background on Artin $L$-functions see Neukirch \cite{N}; see p. 524 for the proof of \eqref{eq:ded_non_abelian}
in particular.)

This formula is the non-abelian generalization of \eqref{eqn:kw}.
In general, it is not known that the $L(s, \rho)$ are ``proper'' $L$-functions
(as defined on \cite[p. 94]{IK} for example)
and in particular that they are holomorphic in the critical strip.
However, this was conjectured by Artin; we refer to this assumption as the Artin Conjecture and
assume its truth in Section \ref{sec_main}.

\begin{remark} As a consequence of Brauer's theorem on group
  characters \cite[p. 522]{N}, It is known that the Artin
  $L$-functions are quotients of Hecke $L$-functions, and therefore
  meromorphic, and this suffices in many applications. For example,
  Lagarias and Odlyzko \cite{LO} used this fact to prove an
  unconditional and effective version of the Chebotarev density
  theorem.
\end{remark}

If $\Gal(L/K)$ is abelian, then the representations are all
one-dimensional, and class field theory establishes that the
characters of $\Gal(L/K)$ coincide with Hecke characters of $L/K$, so
that \eqref{eq:ded_non_abelian} becomes
\begin{equation}\label{eqn:kw2}
\zeta_{L/K}(s) = \zeta_K(s) \prod_{i = 2}^{[L:K]} L(s, \chi_i),
\end{equation}
where the product ranges over Hecke characters of $K$.  As in our
application of \eqref{eqn:kw}, we will argue that there cannot be too
many characters $\chi$ or representations $\rho$ of small conductor
(and, in the latter case, of bounded degree).

We can use \eqref{eq:ded_non_abelian} to derive more general versions of
\eqref{eq:abelian_rd_bound}: It follows \cite[p. 527]{N} from \eqref{eq:ded_non_abelian} that the relative discriminant $\mfd(L/K)$ satisfies the formula
\begin{equation}\label{disc_eqn1}
\mfd(L/K) = \prod_{\rho \in \Irr(G)} \mff(\rho)^{\deg \rho},
\end{equation}
where the ideal $\mff(\rho)$ of $K$ is the Artin conductor associated to $\rho$.

If $L/K$ is abelian then
we can write this as $\mfd(L/K) = \prod_{\chi_i} \mff(\chi_i)$.
Taking norms down to $\Q$ and using the relation (\cite{N}, p. 202)
\begin{equation}\label{eqn:rel_disc}
|\Disc(L)|
= |\Disc(K)|^{[L:K]} \calN_{K/\Q} (\mfd_{L/K}),
\end{equation}
we obtain
\begin{equation}\label{eqn:rd_gen}
\log(\rd_L) = \log(\rd_K) + \frac{1}{[L : \Q]} \sum_i \log( \N_{K/\Q} \mff(\chi_i)).
\end{equation}

If $L/K$ is not necessarily abelian, then
the conductor $q(\rho)$ of $L(s, \rho)$ is related to $\mff(\rho)$ by the formula
\begin{equation}\label{def_conductor}
q(\rho) = |\Disc(K)|^{\deg \rho} \N_{K/\Q} \mff(\rho).
\end{equation}
Taking absolute norms in \eqref{disc_eqn1}, multiplying by
$|\Disc(K)|^d$, and again using \eqref{eqn:rel_disc} we obtain
\begin{equation}\label{disc_eqn3}
|\Disc(L)| = \prod_{\rho \in \Irr(G)} q(\rho)^{\deg \rho}.
\end{equation}

\section{Proof of Theorem \ref{thm_uncond}}\label{sec_elem}
We first prove a lemma bounding some quantities which occur in the
proof.
\begin{lemma}\label{lem:cn_bound}
  For a number field $F$ of degree $f$, the
  following hold:
\begin{enumerate}
\item The number of ideals $\mfa$ of $\mathcal{O}_F$ with $\N(\mfa) <
  Y$ is bounded by $e Y (1 + \log Y)^{f}$.
\item We have $h_F < e |\Disc(F)|^{1/2} \big(1 + \frac{1}{2} \log |\Disc(F)| \big)^f$.
\end{enumerate}
\end{lemma}
\begin{proof}
  This is standard and we give an easy proof inspired by
  \cite[p. 68]{CM}. We have that
$\zeta_{F}(s) = \sum_{n=1}^{\infty} \frac{a_{n}(F)}{n^{s}}$,
where $a_{n}(F)$ is the number of integral ideals of norm $n$ in
$\mathcal{O}_{F}$. The coefficient-wise bound
$\zeta_{F}(s) < \zeta(s)^{f} = \sum_n d_f(n) n^{-s}$ yields that for $\sigma > 1$,
$$\sum_{n < Y} d_f(n) < \sum_n d_f(n) (Y/n)^{\sigma} = Y^{\sigma} \zeta(\sigma)^f.$$ We now choose
$\sigma = 1 + 1/\log Y$, and use the fact that $\zeta(\sigma) < 1 +
\frac{1}{\sigma - 1}$ for $\sigma > 1$.

The second part follows from the classical Minkowski bound (see,
e.g. \cite[Ch. 1.6]{N}), which implies that each ideal class in
$\mathcal{O}_F$ is represented by an ideal $\mfa$ with $\N(\mfa) <
\sqrt{|\Disc(F)|}$.
 \end{proof}

\begin{proof}[Proof of Theorem \ref{thm_uncond}]

The proof is similar to that of \eqref{eqn_simple}, but we will need to work with messier inequalities.

By the character theory remarks after the theorem, $L$ has a subfield $F$, for which $L/F$ is abelian, such that
$[F:K] \leq (r!)^2 < r^{2r}$. We assume that $[L : \Q]$ and therefore $\rd_L$
are bounded below by absolute
constants ($[L : \Q] \geq e^{e^8}$ suffices). 
Depending on the relative sizes of these quantities, we will see that either
\begin{equation}\label{eq:uncond_r_bound}
r \geq C_{5} \frac{ \log \log([L:\Q])}{\log \log \log([L:\Q])} - \log [K:\Q],
\end{equation}
or
\begin{equation}\label{eq:uncond_rd_bound}
\log(\rd_L) \geq C_{6} \log \log([L:\Q]),
\end{equation}
for positive constants $C_{5}$ and $C_{6}$, implying the theorem. There is no obstacle to determining
particular values for these constants, but for simplicity we omit the details.

We begin with the generalization \eqref{eqn:rd_gen} of \eqref{eq:abelian_rd_bound},
which said that
\begin{equation}\label{eqn:rd_gen2}
\log(\rd_L) = \log(\rd_F) + \frac{1}{[L : \Q]} \sum_i \log( \N_{F/\Q} \mff(\chi_i)),
\end{equation}
where $\chi_i$ are distinct Hecke characters of $F$.
The number of characters of conductor $\mfm$ is less than $2^{[F:\Q]} h_F
\N(\mfm)$ \cite[Theorem V.1.7, pg. 146]{milne}, and Lemma
\ref{lem:cn_bound} bounds both $h_F$ and the number of $\mfm$ which
can appear, so that for $Y \geq 1$ the number of characters whose conductor has norm $\leq Y$ is
bounded above by $e^2 Y^2 |\Disc(F)|^{1/2} \big(2 + \log(Y^2 |\Disc(F)|)
\big)^{2 [F : \Q]}$.

Given $[L : \Q]$ and $[F : \Q]$, suppose that $Y > e^{-1} |\Disc(F)|^{-1/2}$ is defined by the equation
\begin{equation}\label{def_y}
\frac{[L : \Q]}{2 [F : \Q]} = e^2 Y^2 |\Disc(F)|^{1/2} \big(2 + \log(Y^2 |\Disc(F)|) \big)^{2 [F : \Q]},
\end{equation} so that
in \eqref{eqn:rd_gen2} there are at least $\frac{[L : \Q]}{2 [F : \Q]}$ characters of conductor $> Y$, so that
\begin{equation}\label{eq:rd_2}
\log(\rd_L) \geq \log(\rd_F) + \frac{1}{2 [F : \Q]} \log Y.
\end{equation}
(Observe that we do not necessarily have $Y > 1$, for example if $L$ is the Hilbert class field
of $F$.)
We divide our analysis of \eqref{eq:rd_2} into three cases
and prove that each implies \eqref{eq:uncond_r_bound} or \eqref{eq:uncond_rd_bound}.

{\bf Large discriminant.}
If $\Disc(F) \geq [L : \Q]^{1/10}$, we ignore \eqref{eq:rd_2} and instead note that
$\log(\rd_F) \geq \frac{1}{10 [F : \Q]} \log([L : \Q])$
and so
\begin{equation}
\log(\rd_L) \geq \frac{1}{10 r^{2r} [K : \Q]} \log([L : \Q]),
\end{equation}
and we obtain at least one of \eqref{eq:uncond_r_bound} and \eqref{eq:uncond_rd_bound} depending on
whether or not $r^{2r} [K : \Q] > (\log [L : \Q])^{1/2}$.

We assume henceforth that $\Disc(F) < [L : \Q]^{1/10}$, which implies that
$[F : \Q] < \frac{1}{10} \log([L : \Q])$ for $[F : \Q] \geq 3$, and write
$Y' := \max(Y, 100)$ and
$Z := Y'^2 |\Disc(F)|^{1/2}$.

{\bf Small discriminant and large degree.}
Assume first that either $Z \leq (2 + 2 \log Z)^{2 [F : \Q]}$ or $Y < 100$.
Applying our upper bounds on $\Disc(F)$, $[F : \Q]$,
and $Z$, we see that
\begin{equation}
[L : \Q]^{4/5} \leq C_{7} \Big(4 \log (Y'^2 |\Disc(F)|)\Big)^{4 [F : \Q]}.
\end{equation}
Taking logarithms
and applying the bound\footnote{
If $Y' = 100$, this follows from $\Disc(F) < [L : \Q]^{1/10}$. If
$Y' = Y$, this follows from \eqref{def_y}.}
$Y'^2 |\Disc(F)| < [L : \Q]$ we obtain
\begin{equation}
\log([L : \Q]) \leq C_{8} [F : \Q] \log \log([L : \Q]),
\end{equation}
so that $\frac{\log([L : \Q])}{\log \log([L : \Q])} \leq C_{9} [K : \Q] r^{2r}$, 
which implies that $r \geq C_{10} \frac{\log \log([L : \Q])}{\log \log \log([L : \Q])} - \log([K : \Q])$.

{\bf Small discriminant and small degree.}  Finally, assume that $Z >
(2 + 2 \log Z)^{2 [F : \Q]}$ and $Y \geq 100$.  Then $[F : \Q] \leq
C_{11} \frac{\log Z}{\log \log Z}$, and our bound on $\Disc(F)$
implies that $\log Z \leq C_{12} \log Y$, so that $[F : \Q] \leq C_{13}
\frac{\log Y}{\log \log Y}$. We thus have $\frac{\log Y}{2 [F : \Q]}
\geq C_{14} \log \log Y$ and so by \eqref{eq:rd_2}
\[
  \log(\rd_{L}) \geq \log(\rd_{F}) + \frac{1}{2 [F : \Q]} \log(Y)
  \geq \log(\rd_{F}) + C_{14} \log \log(Y).
\]
Finally, \eqref{def_y} implies that
$\log Y \leq \frac{1}{2} \log([L : \Q])$, giving us
\begin{equation}
\log(\rd_L) \geq \log(\rd_F) + C_{15} \log \log([L : \Q]).
\end{equation}

This completes our list of cases, and hence the proof.
\end{proof}

\section{Proof of Theorem \ref{thm_main}}
\label{sec_main}
In the proof we will assume familiarity
with Artin $L$-functions and Rankin-Selberg convolutions, as described in
Neukirch \cite{N} (and Section \ref{sec_background}) and Iwaniec-Kowalski \cite{IK} respectively.
We also assume the truth of the Artin Conjecture.
There is no theoretical obstacle to carrying out the methods of this section
without any unproved hypotheses, but when we tried this the
error terms in \eqref{eqn_rs_error} were too large to
be of interest.

As in the non-abelian case, we need to bound the number of possible
$q(\rho)$ of bounded conductor (and now also of bounded
degree). However, in general the representations $\rho$ are not (yet!)
known to correspond to arithmetic objects which might be more easily
counted. Instead, in Proposition \ref{prop_artin_diff} we
(conditionally) bound the number of possible $L$-functions. Assuming
GRH and the Artin Conjecture, we will see that any two such Artin
$L$-functions must have rather different Dirichlet series
representations, because their Rankin-Selberg convolution cannot have
a pole. A pigeonhole-type argument will then allow us to bound the
number of possible $L$-functions.

After proving Proposition \ref{prop_artin_diff} we will conclude as before. In brief, if $L/K$ is a Galois extension of large degree with
many representations of small degree, then many of these representations will have large conductor,
and so $L$ will have large discriminant.

\begin{proposition}\label{prop_artin_diff} Assume that the Artin Conjecture
  and Riemann Hypothesis hold for Artin $L$-functions, and let $\rho$
  and $\rho'$ be distinct irreducible nontrivial representations of
  $\Gal(L/K)$ of degree $r$ and conductor $\leq q$ (as defined in
  \eqref{def_conductor}). The Artin $L$-series of $\rho$ and $\rho'$
have Euler products
\[
  L(s,\rho) = \prod_{\mathfrak{p}} \prod_{i=1}^{d}
  (1 - \alpha_{i,\rho}(\mathfrak{p}) N(\mathfrak{p})^{-s})^{-1},
  \quad 
  L(s,\rho') = \prod_{\mathfrak{p}} \prod_{i=1}^{d}
  (1 - \alpha_{i,\rho'}(\mathfrak{p}) N(\mathfrak{p})^{-s})^{-1}.
\]
Assume that $\log q > r [K : \Q]$.
Then, for $X \geq \Cxbound r^2 \log^2 q$ we have
\begin{equation}\label{eqn-artin-diff}
\sum_{\substack{\N \mfp \in [X, 2X]\\ \mathfrak{p} \text{ unramified }}} \sum_{1 \leq i \leq r}
| \alpha_{i, \rho}(\mfp) -  \alpha_{i, \rho'}(\mfp) |
\geq  \frac{X}{2 r \log X}.
\end{equation}
Furthermore, the number of representations of degree $\leq r$ and
conductor $\leq q$ is
\begin{equation}\label{eqn-diff-repns}
\leq \Crepexponent^{r^3 \log^2(q) [K : \Q]},
\end{equation}
for an absolute constant $\Crepexponent$.
\end{proposition}
\begin{remark} The bound \eqref{eqn-diff-repns} is rather
  simple-minded, and we could remove the factor $[K : \Q]$ by instead
  insisting that $q$ be sufficiently large in terms of $K$.
\end{remark}
\begin{proof}
  This is essentially Proposition 5.22 of Iwaniec and Kowalski
  \cite{IK}; although our conclusion is stronger, our proof is
  essentially the same.

  Consider the tensor product representations $\rho \otimes
  \overline{\rho}$ and $\rho \otimes \overline{\rho'}$, whose
  $L$-functions are equal to the Rankin-Selberg convolutions $L(s,
  \rho \otimes \overline{\rho})$ and $L(s, \rho \otimes
  \overline{\rho'})$ (so that the notation is not ambiguous). A simple
  character-theoretic argument shows that the trivial representation
  does not occur in $\rho \otimes \overline{\rho'}$, while it occurs
  with multiplicity one in $\rho \otimes \overline{\rho}$.  Assuming
  the Artin conjecture, then, $L(s, \rho \otimes \overline{\rho'})$
  and $L(s, \rho \otimes \overline{\rho}) \zeta(s)^{-1}$ are entire
  functions.

  Let $\phi$ be a smooth test function with support in
  $[1, 2]$, image in $[0, 1]$, and $\widehat{\phi}(0) = \int_{1}^{2} \phi(t) dt \in (\frac{3}{4}, 1)$;
  throughout this section, all implied constants (including $\Cxbound$, etc.) depend on our fixed choice of $\phi$.
   Also, let $X \geq 2$, with stricter lower bounds to be imposed later.  Then,
  using the ``explicit formula'' (Theorem 5.11 of \cite{IK}), we find
  (see p. 118 of \cite{IK}) that, assuming GRH and the Artin Conjecture,
\begin{equation}\label{rs_1}
\left|\sum_n \Lambda_{\rho \otimes \overline{\rho'}}(n) \phi (n / X)\right|
\ll \sqrt{X} \log(\mathfrak{q}(\rho \otimes \overline{\rho'})),
\end{equation}
\begin{equation}\label{rs_2}
\left|\sum_n \Lambda_{\rho \otimes \overline{\rho}}(n) \phi (n / X)
 - \widehat{\phi}(0) X \right|
\ll \sqrt{X} \log(\mathfrak{q}(\rho \otimes \overline{\rho})),
\end{equation}
where the coefficients
$\Lambda$ are defined, for any $L$-function $L(f, s)$, by the relation
$$\sum_n \Lambda_f(n) n^{-s} = - \frac{L'}{L}(f, s).$$
Also, $\mathfrak{q}(\rho) = q(\rho) \prod_{j=1}^{r} (|\kappa_{j}| +
3)$ denotes the analytic conductor of $\rho$ defined by equation (5.8)
of \cite{IK}. To bound this analytic conductor, we require information
about the gamma factors of Artin $L$-functions and the conductor
of $\rho \otimes \psi$ (where $\psi = \overline{\rho}$ or $\overline{\rho'}$).

The fact that the $L(s,\rho)$ are factors of the Dedekind zeta function, and the fact that the Dedekind zeta function only has gamma factors $\Gamma\left(\frac{s}{2}\right)$ and $\Gamma\left(\frac{s+1}{2}\right)$ implies that $0 \leq
\kappa_{j} \leq 1/2$ for all $j$. We have that $q(\rho \otimes \psi)
= |\Disc(K)|^{\deg(\rho \otimes \psi)} \N_{K/\Q}(\mathfrak{f}(\rho \otimes \psi))$ where
\[
  \mathfrak{f}(\rho \otimes \psi) = \prod_{\mathfrak{p} \nmid \infty}
  \mathfrak{p}^{\mathfrak{f}_{\mathfrak{p}}(\rho \otimes \psi)},
  \qquad f_{\mathfrak{p}}(\rho \otimes \psi)
  = \sum_{i=0}^{\infty} \frac{g_{i}}{g_{0}} {\rm codim}~V_{\rho \otimes \psi}^{G_{i}}.
\]
Here $V_{\rho \otimes \psi}$ is a vector space affording the representation 
$\rho \otimes \psi$, $G_{i}$ is the $i$th ramification group, and $g_{i} = |G_{i}|$ (this definition is from page 527 of \cite{N}). Noting that
\[
  f_{\mathfrak{p}}(\rho) = \sum_{i=0}^{\infty} \frac{g_{i}}{g_{0}} {\rm codim}~V_{\rho}^{G_{i}}, \text{ and }
  f_{\mathfrak{p}}(\psi) = \sum_{i=0}^{\infty} \frac{g_{i}}{g_{0}} {\rm codim}~V_{\psi}^{G_{i}}.
\]
If $a = {\rm codim}~V_{\rho}^{G_{i}}$ and $b = {\rm codim}~V_{\psi}^{G_{i}}$,
$G_{i}$ fixes a subspace of $V_{\rho \otimes \psi}$
of dimension at least $(r-a)(r-b)$ and it follows from this
that $f_{\mathfrak{p}}(\rho \otimes \psi) \leq r (f_{\mathfrak{p}}(\rho) + f_{\mathfrak{p}}(\psi))$ and
\begin{equation}
  q(\rho \otimes \psi) \leq |\Disc(K)|^{r^{2}} \N_{K/\Q}(\mathfrak{f}(\rho) \mathfrak{f}(\psi)).
\end{equation}
Combining these estimates for the analytic conductor yields the bound
\begin{align*}
  \log(\mathfrak{q}(\rho \otimes \psi))
  &\leq \left(2 r^{2} + r^{2} \log |\Disc(K)| + r \log \N_{K/\Q}(\mathfrak{f}(\rho) \mathfrak{f}(\psi))\right)\\
  &\leq 3 r \log(q).
\end{align*}
Let $\alpha_{i, \rho}$ and $\alpha_{i, \rho'}$ be the Frobenius
eigenvalues for $\rho$ and $\rho'$ respectively, for $1 \leq i \leq
r$. Then, unpacking the definition of the Rankin-Selberg convolution
(or, equivalently, the tensor product representation), we conclude
that
\begin{equation}\label{eqn_rs_error}
\sum_{\N \mfp \in [X, 2X]} \sum_{1 \leq i, j \leq r}
\left|\alpha_{i, \rho}(\mfp) \overline{\alpha_{j, \rho}(\mfp)} - \alpha_{i,
  \rho}(\mfp) \overline{\alpha_{j, \rho'}(\mfp)}\right| \geq
\frac{\widehat{\phi}(0) X}{\log (2X)} - \frac{\CIKerror \sqrt{X}}{\log X} (r \log
q).
\end{equation}

The sum above is over prime ideals; we have removed the prime powers
which contribute (assuming $X$ is large)
$
< \frac{2r^{2} [K : \Q]
  \sqrt{X}}{\log(X)}
< \frac{2r \log(q)
  \sqrt{X}}{\log(X)},
 $
which is contained in the error term above.
Next, we remove the terms coming from ``ramified primes'', those for
which $\alpha_{i,\rho}(\mathfrak{p}) = 0$ for some $i$. These occur precisely
at the primes for which $\mathfrak{p} | \mathfrak{f}(\rho)$ and
the number of these between $X$ and $2X$ is at most $\frac{\log q}{\log X}$.
Noting that for a prime $\mathfrak{p}$ that is unramified in $L/K$,
$|\alpha_{i,\rho}(\mathfrak{p})| = 1$ for all $i$, we get
\begin{equation}
\label{sumisbig}
\sum_{\substack{\N \mfp \in [X, 2X],\\ \mathfrak{p} \text{ unramified }}} \sum_{1 \leq j \leq r}
| \alpha_{j, \rho}(\mfp) -  \alpha_{j, \rho'}(\mfp) |
\geq \frac{\widehat{\phi}(0) X}{r \log (2X)} -
\CIKerrorb \left(\frac{\sqrt{X} (\log q) + r \log q}{\log X}\right).
\end{equation}
If $X \geq \Cxbound r^{2} \log^{2}(q)$ with $\Cxbound =  \max\Big(2^{14}, 100 \CIKerrorb^2\Big)$,
the error term above is $\leq \frac{ X}{5 r \log X}$ and the main term is
$\geq \frac{7X}{10 r \log X}$,
establishing the first part of the proposition.

The second part follows easily from the first. Let $M$ be the number of
primes in the above sum; then, if
\[
  |\alpha_{i,\rho}(\mathfrak{p}) - \alpha_{i,\rho'}(\mathfrak{p})|
  \leq \frac{X}{2 r^{2} \log(X) M}
\]
for all $i$ and unramified $\mathfrak{p}$, then \eqref{eqn-artin-diff} is contradicted.
Note that since $\mathfrak{p}$ is unramified, $\alpha_{i,\rho}(\mathfrak{p})$ and
and $\alpha_{i,\rho'}(\mathfrak{p})$ lie on the unit circle. No more than $2 \pi Y$ 
points can be placed
on the unit circle with pairwise distances at least $1/Y$. Hence,
by the pigeonhole principle, there can be at most
\[
  N = \left(\frac{4 \pi r^{2} \log(X) M}{X}\right)^{r (\pi_{K}(2X) -
\pi_{K}(X))}
\]
Artin $L$-functions of degree $r$ and conductor $\leq q$. Here,
$\pi_{K}(X)$ is the number of prime ideals of $\mathcal{O}_{K}$ of
norm less than or equal to $X$. We have $M \leq \pi_{K}(2X) -
\pi_{K}(X) \leq 2 [K : \Q] X/\log(X)$ and so
\[
  \log(N) \leq \frac{2 r X [K : \Q]}{\log(X)} [\log(8 \pi r^2) + \log([K : \Q])].
\]
We have
\[
\log(X) \geq \log(2^{14}) + 2 \log(r) + 2 \log \log(q) > \log(8 \pi) + 2 \log(r) + \log( [K : \Q]),
\]
and hence
\[
  \log(N) \leq 2 \Cxbound r^{3} \log^{2}(q) [K : \Q].
\]
As $e^{r^{3} \log^{2}(q) [K : \Q]}$ increases rapidly in $r$, the number of
representations of conductor $\leq q$ and degree $r' \leq r$ is bounded by
\[
  \Crepexponent^{r^{3} \log^{2}(q) [K : \Q]}.
\]
for some absolute constant $\Crepexponent$.
\end{proof}

We now prove Theorem~\ref{thm_main} using Proposition~\ref{prop_artin_diff}.

\begin{proof}[Proof of Theorem \ref{thm_main}]

Assume $L/K$ is a Galois
extension whose irreducible complex
representations all have degree at most $r$. Choose the smallest
$A \geq e$ so that
\begin{equation}
\label{defnofA}
  \frac{[L : K]}{2r^{2}}
  \leq \sum_{i=1}^{r} \Crepexponent^{r^{2} (r - i) \log^{2}(A^{2}) [K : \Q]}.
\end{equation}
We wish to estimate $A$ in terms of $|\Disc(L)|$. By
Proposition \ref{prop_artin_diff}, the number of representations with degree $r-i$ and
conductor $\leq A^{\frac{2r}{r-i}}$ is
\[
  \leq \Crepexponent^{(r-i)^{3} \log^{2}(A^{\frac{2r}{r-i}}) [K : \Q]}
  = \Crepexponent^{4 r^{2} (r-i) \log^{2}(A^{2}) [K : \Q]}.
\]
Every other representation has $q(\rho)^{\deg \rho} \geq A^{2r}$.
There are at least $[L:K]/2r^{2}$ of these, and so \eqref{disc_eqn3} gives
\[
  |\Disc(L)| = \prod_{\rho \in \Irr(\Gal(L/K))}
  q(\rho)^{\deg \rho}
  \geq \left(A^{2r}\right)^{\frac{[L:K]}{2r^{2}}}
  = A^{\frac{[L:K]}{r}}.
\]
Thus, $\log(A) \leq \frac{r}{[L:K]} \log(|\Disc(L)|) = r \cdot [K : \Q]
\log(\rd_{L})$.

Equation~\eqref{defnofA} gives
\[
  \frac{[L:K]}{2r^{2}}
  \leq 2\Crepexponent^{r^{3} \log^{2}(A^{2}) [K:\Q]},
\]
enlarging $\Crepexponent$ if necessary so that $\Crepexponent^{e^2} > 2$,
which gives
\begin{align*}
  \log([L:K]) & \leq \log(4 r^2) + r^{3} \log^{2}(A) [K : \Q] \log \Crepexponent \\
  & \leq \left(\log(4) + \log \Crepexponent\big) \cdot r^{5} [K : \Q]^{3} \log(\rd_{L}\right)^{2}.
\end{align*}
Hence, there is an absolute constant $\Ccond$ so that $r \geq \Ccond
\frac{\log([L:K])^{1/5}}{[K:\Q]^{3/5} \log(\rd_{L})^{2/5}}$.
\end{proof}

Finally, we prove Corollary~\ref{unramified} using Proposition
~\ref{prop_artin_diff}.

\begin{proof}
Let $L$ be the compositum of the $L_{i}$. Then, we have
\[
  \Gal(L/K) = \prod_{i=1}^{N} \Gal(L_{i}/K)
\]
From Theorem 4.21 of \cite{I}, if $\rho_{i} : \Gal(L_{i}/K) \to \GL_{r}(\C)$
is an irreducible representation, then the map $\tilde{\rho}_{i}(g) =
\rho_{i}(g|_{L_{i}})$ is also an irreducible representation of $\Gal(L/K)$
which is distinct from $\tilde{\rho}_{j}$ for $i \ne j$. All of these
representations have conductor $q = |\Disc(K)|^{r}$ and Proposition~\ref{prop_artin_diff} implies that there is an absolute constant $\Ccor$ so that
\[
  \log(N) \leq \Ccor r^{5} \log^{2}(|\Disc(K)|) [K : \Q],
\]
as desired.
\end{proof}


\begin{thebibliography}{99}

\bibitem{CM}
A. C. Cojocaru and M. R. Murty,
{\itshape An introduction to sieve methods and their applications},
Cambridge University Press, Cambridge, 2006.

\bibitem{EV}
J. Ellenberg and A. Venkatesh,
\emph{The number of extensions of a number field with fixed degree
and bounded discriminant}, Ann. of Math. (2) 163 (2006), no. 2, 723--741.

\bibitem{GS}  E. Golod and I. Shafarevich, \emph{On class field towers}, Izv. Akad. Nauk SSSR \textbf{28} (1964), 261--272 [In
Russian]; English transl. in Amer. Math. Soc. Transl., Vol. 48, pp.91--102, American Math Society,
Providence, RI, 1965.

\bibitem{HM1} F. Hajir and C. Maire,
\emph{Tamely ramified towers and discriminant
bounds for number fields}, Compositio Math. \textbf{128} (2001), no. 1, 35--53.

\bibitem{HM2}  F. Hajir and C. Maire,
\emph{Tamely ramified towers and discriminant
bounds for number fields. II}, J. Symbolic Comput. \textbf{33} (2002), no. 4, 415--423.

\bibitem{I} I. M. Isaacs,
\emph{Character theory of finite groups},
AMS Chelsea Publishing, Providence, RI, 2006.

\bibitem{IK} H. Iwaniec and E. Kowalski,
\emph{Analytic number theory},
American Mathematical Society, Providence, RI, 2004.

\bibitem{JR} J. Jones and D. Roberts,
\emph{Galois number fields with small root discriminant},
J. Number Theory \textbf{122} (2007), no. 2, 379--407.

\bibitem{Leshin} J. Leshin,
\emph{Solvable number field extensions of bounded root discriminant},
Proc. Amer. Math. Soc. \textbf{141} (2013), no. 10, 3341-3352.

\bibitem{LO} J. Lagarias and A. Odlyzko,
\emph{Effective versions of the Chebotarev density theorem},
Algebraic number fields: L-functions and Galois properties (Proc. Sympos., Univ. Durham, Durham, 1975), pp. 409--464,
Academic Press, London, 1977.

\bibitem{martinet}
J. Martinet,
\emph{
Petits discriminants},
Ann. Inst. Fourier \textbf{29} (1979), 159--170.

\bibitem{MV}
P. Michel and A. Venkatesh,
\emph{On the dimension of the space of cusp forms associated to
2-dimensional complex Galois representations}, Int. Math. Res. Not. (2002),
no. 38, 2021--2027.

\bibitem{milne} J. S. Milne,
\emph{Class field theory (v4.01)},
available at \url{www.jmilne.org/math/},
2011.

\bibitem{N} J. Neukirch,
\emph{Algebraic number theory},
Springer-Verlag, Berlin, 1999.

\bibitem{O1976} A. Odlyzko,
\emph{Lower bounds for discriminants of number fields}, Acta. Arith. \textbf{29} (1976), no. 3, 275--297.

\bibitem{O} A. Odlyzko,
\emph{Bounds for discriminants and related estimates for class numbers, regulators,
and zeros for zeta functions: a survey of recent results},
S\'em. Th\'eor. Nombres Bordeaux (2) \textbf{2} (1990), no. 1, 119--141.

\bibitem{poitou} G. Poitou,
\emph{Minorations de discriminants (d'apr\'es A. M. Odlyzko)}, S\'eminaire Bourbaki, Vol.
1975/76, 28\`eme ann\'ee, Exp. No. 479, pp. 136-153, Lecture Notes in Math. 567, Springer, 1977.

\bibitem{Stark} H. Stark,
\emph{Some effective cases of the Brauer-Siegel theorem},
Invent. Math. \textbf{23} (1974), 135--152.

\bibitem{V} J. Voight,
\emph{Enumeration of totally real number fields of bounded root discriminant},
Algorithmic number theory, 268--281,
Lecture Notes in Comput. Sci., 5011, Springer, Berlin, 2008.

\end{thebibliography}
\end{document}